\documentclass[1 [leqno,11pt]{amsart}
\usepackage{amssymb, amsmath}

\let\pa=\partial

\let\s=\sigma

\let\f=\frac

\let\D=\Delta
\let\wt=\widetilde
\let\wh=\widehat
\def\om{\omega^{\beta}}

\def\dive{\mathop{\rm div}\nolimits}

\def\ucurl{u^{\beta^\perp}_{{\rm curl}}}
\def\udiv{u^{\beta^\perp}_{{\rm div}}}
\def\uh{u^{\beta^\perp}}
\def\nablah{\nabla_{\beta^\perp}}
\def\Dh{\D_{\beta^\perp}}
\def\pb{\partial_\beta}
\def\ub{u^\beta}

\def\dhk{\Delta_k^{\beta^\perp}}
\def\dvl{\Delta_\ell^{\beta}}
\def\dhkp{\Delta_{k'}^{\beta^\perp}}
\def\dvlp{\Delta_{\ell'}^{\beta}}
\def\dvlpp{\Delta_{\ell''}^{\beta}}
\def\wtdhkp{\wt\Delta_{k'}^{\beta^\perp}}
\def\wtdvlp{\wt\Delta_{\ell'}^{\beta}}
\def\Shkp{S_{k'-1}^{\beta^\perp}}
\def\Shk{S_{k-1}^{\beta^\perp}}
\def\Svlp{S_{\ell'-1}^{\beta}}

\def\dB{\dot{B}}
\def\dH{\dot{H}}

\def\cA{{\mathcal A}}
\def\cB{{\mathcal B}}
\def\cC{{\mathcal C}}

\def\cF{{\mathcal F}}

\def\cH{{\mathcal H}}

\def\Z{\mathop{\mathbb Z\kern 0pt}\nolimits}
\def\N{\mathop{\mathbb N\kern 0pt}\nolimits}
\def\Q{\mathop{\mathbb Q\kern 0pt}\nolimits}
\def\R{\mathop{\mathbb R\kern 0pt}\nolimits}

\def\eqdef{\buildrel\hbox{\footnotesize def}\over =}
\def\eqdefa{\buildrel\hbox{\footnotesize def}\over =}
\newcommand{\andf}{\quad\hbox{and}\quad}
\newcommand{\with}{\quad\hbox{with}\quad}
\def\Supp{\mathop{\rm Supp}\nolimits\ }

\newcommand{\beq}{\begin{equation}}
\newcommand{\eeq}{\end{equation}}
\newcommand{\ben}{\begin{eqnarray}}
\newcommand{\een}{\end{eqnarray}}
\newcommand{\beno}{\begin{eqnarray*}}
\newcommand{\eeno}{\end{eqnarray*}}

\newtheorem{defi}{Definition}[section]
\newtheorem{thm}{Theorem}[section]
\newtheorem{lem}{Lemma}[section]
\newtheorem{rmk}{Remark}[section]
\newtheorem{cor}{Corollary}[section]
\newtheorem{prop}{Proposition}[section]

\numberwithin{equation}{section}
\begin{document}

\title[One component regularity criteria for 3-D Navier-Stokes equations]
{On the one time-varying component regularity criteria for 3-D Navier-Stokes equations}

\author[Y. Liu]{Yanlin Liu}
\address[Y. Liu]{School of Mathematical Sciences,
Laboratory of Mathematics and Complex Systems,
MOE, Beijing Normal University, 100875 Beijing, China.}
\email{liuyanlin@bnu.edu.cn}

\author[P. Zhang]{Ping Zhang}
\address[P. Zhang]{Academy of Mathematics $\&$ Systems Science
	and Hua Loo-Keng Center for Mathematical Sciences, Chinese Academy of
	Sciences, Beijing 100190, CHINA, and School of Mathematical Sciences,
	University of Chinese Academy of Sciences, Beijing 100049, China.} \email{zp@amss.ac.cn}

\date{\today}

\begin{abstract}
 In this paper, we consider the  one time-varying component regularity criteria for local strong solution of 3-D Navier-Stokes
equations. Precisely, if $\beta(t)$ is a piecewise $H^1$ unit vector from $[0,T] $  to $\Bbb{S}^2$ with finitely many jump
discontinuities, we prove that if $\int_0^{T}\|u(t)\cdot \beta(t)\|_{\dH^{\f32}(\R^3)}^2\,dt<\infty,$
then the solution $u$ can be extended beyond the time $T.$ Compared with the previous results \cite{CZ5,CZZ,LeiZhao}
concerning one-component regularity criteria,
here the unit vector $\beta(t)$ varies with time variable.

\end{abstract}
\maketitle


\setcounter{equation}{0}
\section{Introduction}
In this paper, we investigate the necessary condition for the
breakdown of regularity of strong solutions to
 $3$-D incompressible Navier-Stokes equations:
\begin{equation*}
(NS)\qquad \left\{\begin{array}{l}
\displaystyle \pa_t u+u\cdot\nabla u-\D u=-\nabla P, \qquad (t,x)\in\R^+\times\R^3, \\
\displaystyle \dive u = 0, \\
\displaystyle  u|_{t=0}=u_0,
\end{array}\right.
\end{equation*}
where $u$ stands for the  fluid  velocity and
$P$ for the scalar pressure function, which guarantees the divergence free condition of the velocity field.

In seminal paper \cite{lerayns}, among other important results, Leray
 proved the local existence and uniqueness of the strong solution to $(NS)$: $u\in C([0,T];H^1(\R^3))\cap L^2(]0,T[; \dot{H}^2(\R^3))$\footnote{Throughout this paper,
we use $\dH^s(\R^3)$ (resp. $H^s(\R^3)$) to denote
 homogeneous (resp. inhomogeneous) Sobolev space with norm defined by
$$\|a\|_{\dH^s(\R^3)}^2\eqdef\int_{\R^3}|\xi|^{2s}|\wh a(\xi)|^2\,d\xi,
\quad\Bigl(~\text{resp.}~\|a\|_{H^s(\R^3)}^2\eqdef
\int_{\R^3}(1+|\xi|)^{2s}|\wh a(\xi)|^2\,d\xi~\Bigr).$$}.
And the well-known Ladyzhenskaya-Prodi-Serrin criteria claims that
if the maximal existence time $T^\ast$ of a strong
solution $u$ is finite, then there holds
\begin{equation}\label{blowupLPS}
\int_0^{T^\ast}\|u(t)\|_{L^q(\R^3)}^p\,dt=\infty,
\quad\forall\  p\in[2,\infty[\with\f2p+\f3q=1.
\end{equation}
In view of Sobolev embedding theorem,
we can derive a weaker form of \eqref{blowupLPS} that
\begin{equation}\label{blowupbasic}
\int_0^{T^\ast}\|u(t)\|_{\dH^{\f12+\f2p}(\R^3)}^p\,dt=\infty,
\quad\forall\  p\in[2,\infty[,
\end{equation}
which was in fact proved by Fujita and Kato in \cite{fujitakato} for
the mild solutions constructed there.

It is worth mentioning that, the end-point case of \eqref{blowupLPS}
when $p=\infty$, namely
\begin{equation}\label{blowupendpoint}
\limsup_{t\rightarrow T^\ast}
\|u(t)\|_{L^3(\R^3)}=\infty,
\end{equation}
is much deeper, which is proved by Escauriaza, Seregin and \v{S}ver\'{a}k
 in \cite{ISS}
by using the technique of backward uniqueness and unique continuation.
One can also check \cite{KK11} for a different approach by using profile decomposition,
and \cite{Tao} for a quantitative blow-up rate.

Before proceeding, let us recall the scaling property of $(NS)$, which means that for any solution $u$ of $(NS)$ on $[0,T]$ and any parameter $\lambda>0$, $u_\lambda$ defined by
\beq \label{scale} u_\lambda(t,x)\eqdef\lambda u(\lambda^2 t,\lambda x) \eeq
is also a solution of $(NS)$ on $[0,T/\lambda^2].$
As Leray emphasized in \cite{lerayns} that all the reasonable estimates to $(NS)$ should be invariant under the scaling transformation
\eqref{scale}. And it is not difficult to verify that, the criteria \eqref{blowupLPS}-\eqref{blowupendpoint} are all scaling invariant.
\smallskip

Next, we review some remarkable blow-up criteria that  involves only
one entry of $u$ or $\nabla u$. The first result in this direction is due to Neustupa and Penel [21]. Kukavica and Ziane proved in \cite{Kukavica} that
\begin{equation}\label{KZ}
T^\ast<\infty\Longrightarrow \int_0^{T^\ast}\|u^3(t)\|_{L^q}^p\,dt=\infty \with \f2p+\f3q\leq\f58
\andf q\in \bigl[{24}/5,\infty\bigr].
\end{equation}
After this, there are numerous works trying to refine the range of $(p,q)$,
here we only list \cite{CT08,CT11,ZP10} for instance.
However, it is worth mentioning that, the norms involved in these criteria are all far from being scaling invariant.
Until very recently, Chae and Wolf \cite{CW21} made an important progress
to generalize \eqref{KZ}  for any $(p,q)$ satisfying $\f2p+\f3q<1$ with $q\in ]3,\infty].$
 Laterly, by using the Lorentz space
$L^{q,1}_t(L^p)$ instead of the Lebesgue space $L^{q}_t(L^p)$,
\cite{WWZ} finally attained the scaling invariant case with $\f2p+\f3q=1$.
Observing that the results in \cite{CW21,WWZ} are very close to
the one-component version of \eqref{blowupLPS}.

On the other hand, as far as we know, the first scaling invariant
regularity criteria for $(NS)$ that involves only one component of $u$
was given by Chemin and the second author in \cite{CZ5}.
Precisely, they proved the one-component version of \eqref{blowupbasic}:
\begin{equation}\label{blowupCZ5}
T^\ast<\infty\Longrightarrow
\int_0^{T^\ast}\|u^3(t)\|_{\dH^{\f12+\f2p}}^p\,dt
=\infty,\quad\forall\ p\in]4,6[.
\end{equation}
Later, \cite{CZZ} generalized \eqref{blowupCZ5} to $p\in]4,\infty[$,
and \cite{LeiZhao} dealt with the remaining case for $p\in[2,4]$.

We mention that, due to the Galilean invariance  of the system $(NS)$, all the one-component criteria listed above hold not only for $u^3$, but also for $u\cdot e$, where $e$ can be any unit constant vector in $\R^3.$ However, it seems that there is
  no work investigating the time-dependent unit vector case.
   And this is  the aim of this paper.

Our main result states as follows:

\begin{thm}\label{thm1}
{\sl If a strong solution $u$ to $(NS)$ blows up at some finite time $T^\ast$,
then for any $\beta(t)\in\Omega(T^\ast)$, there holds
\begin{equation}\label{blowupthm1}
\int_0^{T^\ast}\|u(t)\cdot \beta(t)\|_{\dH^{\f32}(\R^3)}^2\,dt=\infty.
\end{equation}
Here $\Omega(T)$ is a subset of time-dependent unit vector fields
defined as follows:
\begin{equation}\begin{split}\label{defomt}
\Omega(T)\eqdef\Bigl\{\beta:[0,T[&\rightarrow\mathbb{S}^2\,\big|
\text{ $\beta(t)$  has finitely many jump discontinuities:}\\
& T_1,\cdots, T_n,\  \text{on $]0,T[$\ with
$\beta'\in L^2(]T_{i-1},T_i[)$ for each $i\in [1,n]$}\Bigr\}.
\end{split}\end{equation}
}\end{thm}

\begin{rmk}
It is interesting to observe that the one-component criteria
indicates that if $T^\ast$ is finite,
then $u$ blows up in every direction simultaneously. This reflects
the isotropic property of the viscous incompressible fluids.
While comparing to all the previous results,
Theorem \ref{thm1} allows us to take this component differently in different time.
In this sense, Theorem \ref{thm1} is more convincing that the possible blow-up can happen only isotropically.

On the other hand, we think it could be a more exciting result, and of course much more challenging, to drop all the smoothness assumptions on $\beta$ in Theorem \ref{thm1}. Precisely, we can raise the following question:

If a strong solution $u$ to $(NS)$ blows up at some finite time $T^\ast$, then can we prove
$$\int_0^{T^\ast}\|u(t)\cdot \beta(t)\|_{\dH^{\f32}(\R^3)}^2\,dt=\infty,
\quad\forall\ \beta:[0,T^\ast]\rightarrow\mathbb{S}^2?$$
In particular, does there necessarily hold
\footnote{One may compare this with \eqref{blowupCZ5}, which asserts that $$\min\Bigl\{\int_0^{T^\ast}
\|u^1(t)\|_{\dH^{\f32}(\R^3)}^2\,dt,
~\int_0^{T^\ast}\|u^2(t)\|_{\dH^{\f32}(\R^3)}^2\,dt,
~\int_0^{T^\ast}\|u^3(t)\|_{\dH^{\f32}(\R^3)}^2\,dt
\Bigr\}=\infty.$$}
$$\int_0^{T^\ast}\min\Bigl\{
\|u^1(t)\|_{\dH^{\f32}(\R^3)}^2,
~\|u^2(t)\|_{\dH^{\f32}(\R^3)}^2,
~\|u^3(t)\|_{\dH^{\f32}(\R^3)}^2\Bigr\}\,dt=\infty?$$
\end{rmk}
\smallskip

Let us end this section with some notations that we shall use throughout this paper.

\noindent{\bf Notations:}
We denote $C$ to be an absolute constant
which may vary from line to line.
And $a\lesssim b$ means that $a\leq Cb$.
$\cF a$ or $\widehat{a}$ denotes the Fourier transform of $a$,
while $\cF^{-1} a$ denotes its inverse.
For a Banach space $B$, we shall use the shorthand $L^p_T(B)$ for $\bigl\|\|\cdot\|_B\bigr\|_{L^p(]0,T[)}$.
And $(a,b)_{\cH}$ designates
the inner product in the Hilbert space $\cH$.

\section{An iteration lemma}

This section is devoted to the study of a common differential inequality,
which might be of independent interest.
Let us consider the following differential inequality for $f(t)\geq 0$:
\begin{equation}\label{2.1}
\left\{\begin{array}{l}
\displaystyle \f{d}{dt}f(t)\leq \f{M}{\s}f^{1+\s}(t)\phi(t),\\
\displaystyle  f|_{t=0}=f_0,
\end{array}\right.
\end{equation}
where $\s,~M$ and $f_0$ are some positive constants,
$\phi$ is some non-negative function which satisfies
\begin{equation}\label{2.2}
\Phi(t',t)\eqdef\int_{t'}^t\phi(s)\,ds<\infty,\quad\forall\ 0\leq t'<t\leq T.
\end{equation}

This kind of differential inequality is often encountered  in the study of PDE.
One may expect to estimate the solution of \eqref{2.1} through Gronwall's type argument
for $\s$ sufficiently small.
However, due to the appearance of $\s^{-1}$
in the coefficient of \eqref{2.1}, it  allows more rapid growth.
Indeed we deduce from
\eqref{2.1} that
$$\f{d}{dt}f^{-\s}(t)=-\s f^{-1-\s}(t)\f{d}{dt}f(t)
\geq -M\phi(t).$$
Integrating the above inequality over $[0,t]$ gives
$$f(t)\leq\bigl(f_0^{-\s}-M\Phi(0,t)\bigr)^{-\f{1}{\s}},$$
which can not rule out the possibility that the solution $f(t)$
may blow up at some finite time $t$ in case
$\Phi(0,t)\geq M^{-1}f_0^{-\sigma}$.

This indicates that if we wish to control $f$ on the whole time interval $[0,T]$,
it is crucial to require some smallness condition for $\Phi(0,T)$.
Unfortunately, in most cases we  only have  the boundedness of $\Phi(0,T)$.
Yet this intuition motivates us to propose the following iteration method to treat the
differential inequality of the type \eqref{2.1}.

\begin{lem}\label{lem2.1}
{\sl\begin{enumerate}
\item[(i)] Let $\{\s_k\}_{k=1}^\infty$ be  a decreasing sequence with
\begin{equation}\label{2.3}
\s_1>\s_2>\cdots>0,\andf \lim_{k\to \infty}\s_k = 0.
\end{equation} Let
 $f$ be a non-negative function which satisfies
\begin{equation}\label{2.4}
\left\{\begin{array}{l}
\displaystyle \f{d}{dt}f(t)\leq \f{M}{\s_k}f^{1+\s_k}(t)\phi(t),
\quad\forall\ k\in\N^+,\\
\displaystyle  f|_{t=0}=f_0,
\end{array}\right.
\end{equation}
where $M$ is an absolute positive constant
which does not depend on $\s_k$,
and $\phi(t)$ satisfies \eqref{2.2}.
Then there exists some constant
$A>0$ depending only on $M,~f_0,~\Phi(0,T)$ and the sequence $\{\s_k\}_{k=1}^\infty$
such that
\begin{equation}\label{2.5}
f(t)\leq A,\quad\forall\ t\in[0,T].
\end{equation}

\item[(ii)] If  there exists some $\delta>0$ such that $f$ satisfies \eqref{2.1}
for every $\s\in]0,\delta]$, then the bound $A$ in \eqref{2.5}
can be chosen to be $f_0\bigl(2^{\f1\delta}+f_0\bigr)
^{2^{2+16M\Phi(0,t)}}$.
\end{enumerate}
}\end{lem}

\begin{proof} (i)
In view of \eqref{2.3}, up to a subsequence, we may assume
\begin{equation}\label{2.6}
f_0^{\s_1}\leq2,\andf 0<\s_{k+1}\leq2^{-1}\s_k,\quad\forall\ k\in\N^+.
\end{equation}

While due to $\Phi(0,T)<\infty$, we can divide $[0,T]$
into $n=[16M\Phi(0,t)]+1$ subintervals with: $0=T_0<T_1<\cdots<T_n=T,$
such that
$$\int_{T_{i-1}}^{T_i}\phi(s)\,ds<\f{1}{16M},\quad\forall\ 1\leq i\leq n,$$
which together with the assumption \eqref{2.6} implies that
\begin{equation}\label{2.7}
4Mf_0^{\s_k}\int_{T_{i-1}}^{T_i}\phi(s)\,ds<\f12,\quad\forall\
1\leq i\leq n,~\forall\ k\in\N^+.
\end{equation}

In the following, we shall prove by induction that,
for any $1\leq i\leq n$, $f$ satisfies
\begin{equation}\label{2.8}
f^{\s_i}(t)\leq 4f^{\s_i}_0,\quad\forall\
t\in[T_{i-1},T_i].
\end{equation}

\noindent{\bf Step 1.}
Observing that for the case when $i=1$,
 $f^{\s_1}(0)=f^{\s_1}_0$, we  define
$$T_1^\star\eqdef\sup\left\{\mathfrak{T}\in]0,T_1]\,\big|\,
f^{\s_1}(t)\leq 4f^{\s_1}_0,\quad\forall\
t\in[0,\mathfrak{T}]\right\}.$$
Then for any $t\in[0,T_1^\star]$,
we get, by using the inequality \eqref{2.4} with $k=1,$ that
$$\f{d}{dt}f(t)\leq \f{M}{\s_1}f(t)f^{\s_1}(t)\phi(t)
\leq\f{4Mf^{\s_1}_0}{\s_1}f(t)\phi(t).$$
By applying Gronwall's inequality and using \eqref{2.7}, we infer
$$f(t)\leq f_0\exp\bigl(\f{4Mf^{\s_1}_0\Phi(0,t)}{\s_1}\bigr)
<f_0 e^{\f{1}{2\s_1}},\quad\forall\ t\in[0,T_1^\star],$$
which implies
$$f^{\s_1}(t)<\sqrt{e}f_0^{\s_1},\quad\forall\ t\in[0,T_1^\star].$$
This contradicts with the definition of $T_1^\star$, unless $T_1^\star=T_1,$
which leads to \eqref{2.8} for $i=1$.

\noindent{\bf Step 2.}
Let us assume that \eqref{2.8} holds  for $1\leq i\leq k-1$ with some $k\geq2$,
we  aim  to prove \eqref{2.8} for $i=k$.
In particular, it follows from  the case when $i=k-1$ that
$$f^{\s_{k-1}}(T_{k-1})\leq 4f^{\s_{k-1}}_0,$$
which together with the assumption: $\s_k\leq2^{-1}\s_{k-1},$ ensures that
\begin{equation}\label{2.9}
f^{\s_k}(T_{k-1})\leq 4^{\f{\s_k}{\s_{k-1}}}f^{\s_k}_0
\leq2f^{\s_k}_0.
\end{equation}
Thanks to \eqref{2.9}, we  define
$$T_k^\star\eqdef\sup\bigl\{\ \mathfrak{T}\in]T_{k-1},T_k]\,\big|\,
f^{\s_k}(t)\leq 4f^{\s_k}_0,\quad\forall\
t\in[T_{k-1},\mathfrak{T}]\bigr\}.$$
Then for any $t\in[T_{k-1},T_k^\star]$,
we get, by using the inequality \eqref{2.4}, that
$$\f{d}{dt}f(t)\leq \f{M}{\s_k}f(t)f^{\s_k}(t)\phi(t)
\leq\f{4Mf^{\s_k}_0}{\s_k}f(t)\phi(t).$$
By applying Gronwall's inequality
and using \eqref{2.7},~\eqref{2.9}, we infer
$$f(t)\leq f(T_{k-1})\exp\bigl(\f{4Mf^{\s_k}_0\Phi(T_{k-1},t)}{\s_k}\bigr)
<2^{\f{1}{\s_k}}f_0 e^{\f{1}{2\s_k}},\quad\forall\ t\in[T_{k-1},T_k^\star],$$
which implies
$$f^{\s_k}(t)<2\sqrt{e}f_0^{\s_k},\quad\forall\ t\in[T_{k-1},T_k^\star].$$
This contradicts with the definition of $T_k^\star$, unless $T_k^\star=T_k,$
so that we  proved \eqref{2.8} for $i=k.$ Then by induction, \eqref{2.8} holds for every $1\leq i\leq n$,
which implies the desired estimate \eqref{2.5}.
\smallskip

\noindent (ii) If there exists some $\delta>0$ such that $f$ satisfies \eqref{2.1}
for every $\s\in]0,\delta]$, then we can choose
$\{\s_k\}_{k=1}^\infty\subset]0,\delta]$ with
$$\s_1\eqdefa \min\bigl\{\log_{1+f_0}2,\delta\bigr\}>0,\andf \s_k\eqdefa 2^{-k+1}\s_1.$$
Then it follows from \eqref{2.8} and the choice of
$n=[16M\Phi(0,t)]+1$ that
$$f(t)\leq 4^{\f{1}{\s_n}}f_0
=4^{\f{2^{[16M\Phi(0,t)]}}{\s_1}}f_0
\leq f_0\bigl(2^{\f1\delta}+f_0\bigr)
^{2^{1+[16M\Phi(0,t)]}},\quad\forall\ t\in[0,T].$$
This completes the proof of  Lemma \ref{lem2.1}.
\end{proof}

\begin{cor}\label{cor2.1}
{\sl Under the assumption of Lemma \ref{lem2.1},
if $g:[0,T]\rightarrow\R^+$ satisfies
\begin{equation}\label{2.10}
\left\{\begin{array}{l}
\displaystyle \f{d}{dt}g(t)\leq \f{M}{\s_k}g^{1+\s_k}(t)\phi(t)
+M_1g(t)V_1(t)+M_2V_2(t),\\
\displaystyle g|_{t=0}=g_0
\end{array}\right.
\end{equation}
 for every $k\in\N^+$,
where $M_1$ and $M_2$ are some nonnegative constants,
$V_1(t)$ and $V_2(t)$ are  nonnegative functions satisfying
$$\int_0^T\bigl(V_1(t)+V_2(t)\bigr)\,dt<\infty,$$
then $g$  is uniformly bounded  on $[0,T]$.
}\end{cor}

\begin{proof}
Let us introduce $h:[0,T]\rightarrow\R^+$ as
$$h(t)\eqdef \bigl(1+g(t)\bigr)
\exp\Bigl(-M_1\int_0^t V_1(s)\,ds-M_2\int_0^tV_2(s)\,ds\Bigr).$$
Then it is easy to observe that $h(t)$ satisfies
\begin{equation*}
\left\{\begin{array}{l}
\displaystyle \f{d}{dt}h(t)
\leq \f{M\exp\bigl(\s M_1\int_0^T V_1(s)\,ds+\s M_2\int_0^TV_2(s)\,ds\bigr)}{\s}
h^{1+\s}(t)\phi(t),\\
\displaystyle h|_{t=0}=g_0,
\end{array}\right.
\end{equation*}
which is of the same form as \eqref{2.1}.
Then  Corollary \ref{cor2.1} follows from Lemma \ref{lem2.1}.
\end{proof}

\section{The functional spaces and some technical lemmas}

In this section,  we shall first introduce the functional spaces
that we are going to use in the following context, and then present some technical lemmas.

As we are going to study the one-component regularity criteria,
it is natural to use spaces that are different in the direction of $\beta(t)$ and
the directions that are perpendicular to $\beta(t).$

\begin{defi}\label{XY}
{\sl For any unit vector $\beta(t)\in\Bbb{S}^2$, we use $\R^2_{\beta^\perp}$
to denote the plane orthogonal to $\beta(t)$,
and $\R_{\beta}$ to denote the line parallel with $\beta(t)$.
For any Banach spaces $X$ and $Y$ on $\R^2_{\beta^\perp}$
and $\R_{\beta}$ respectively, we designate the time-dependent mixed space
$X_{\R^2_{\beta^\perp}}\left(Y_{\R_{\beta}}\right)$ as
$$\|f\|_{X_{\R^2_{\beta^\perp}}\left(Y_{\R_{\beta}}\right)}
\eqdef\|f\|_{X(\R^2_{\beta^\perp};Y(\R_{\beta}))}<\infty.$$
In particular, $L^p_{\beta^\perp}(L^q_\beta)$ denotes $L^p(\R^2_{\beta^\perp}; L^q(\R_\beta)).$
And for any $s_1,s_2\in\R$, we also denote
the anisotropic Sobolev space
$\dH^{s_1}_{\beta^\perp}(\dH^{s_2}_\beta)$ briefly as
$\dH^{s_1,s_2}_{\beta}$, whose norm is given by
$$\|a\|^2_{\dH^{s_1,s_2}_{\beta}} \eqdef \int_{\R^3}
|\xi\times\beta(t)|^{2s_1}|\xi\cdot\beta(t)|^{2s_2}
|\wh a(\xi)|^2\,d\xi.$$
}\end{defi}

Noticing that the norm $\dH^{s_1,s_2}_{\beta(t)}$
actually varies with time $t$, hence it seems not convenient to perform $\dH^{s_1,s_2}_{\beta(t)}$ estimate to the solution $u$ of $(NS)$. Fortunately, for any time $t$, it follows from
Plancherel's identity  that $\dH^{0,0}_{\beta(t)}=L^2$,
while the obvious fact that: $|\xi\times\beta|\leq|\xi|$ and
$|\xi\cdot\beta|\leq|\xi|,$ ensures the following embedding inequalities:
\begin{equation}\label{embeta}
\|a\|_{\dH^{s_1,s_2}_{\beta(t)}}\leq\|a\|_{\dH^{s_1+s_2}},
\andf\|a\|_{\dH^{-s_1,-s_2}_{\beta(t)}}
\geq\|a\|_{\dH^{-s_1-s_2}},
\quad\forall\ s_1\geq0,~s_2\geq0.
\end{equation}

Before proceeding, let us present the explicit coordinate basis for
the plane orthogonal to $\beta(t)$, which will make our statement much easier.

\begin{lem}\label{lemTNB}
{\sl For any finite $T>0$ and any $\beta\in \Omega(T)$,
there exists $\tau,~\nu\in\Omega(T)$ such that
\begin{equation}\label{TNB}
\tau\cdot\nu=\nu\cdot\beta=\beta\cdot\tau=0,\andf
\tau\cdot(\nu\times\beta)=1.
\end{equation}
Moreover, for any $t\in]0,T[,$
$\tau'(t)$ and $\nu'(t)$ exist whenever $\beta'(t)$ exists, and there holds
$$|\tau_i'(t)|+|\nu_i'(t)|\leq C|\beta_i'(t)|.$$
}\end{lem}

\begin{proof}
In view of the definition of $\Omega(T)$ in \eqref{defomt},
we can find a partition of $[0,T]$ with $0=t_0<t_1<\cdots<t_n=T$
so that $\beta'(t)\in L^2(]t_{i-1},t_i[)$ and
\begin{equation}\label{L1beta'}
\|\beta'\|_{L^1(]t_{i-1},t_i[)}
\leq(t_i-t_{i-1})^{\f12}\|\beta'\|_{L^2(]t_{i-1},t_i[)}
<\f15.
\end{equation}
Let us denote $\beta_i(t)$ to be the restriction of $\beta(t)$ on $]t_{i-1},t_i[$.
Noticing that $\beta_i(t_{i-1})=(\beta_i^1,\beta_i^2,\beta_i^3)(t_{i-1})$ is a unit vector,
at least one of its component has absolute value less than $\f35$.
Without loss of generality, we may assume that $|\beta_i^3(t_{i-1})|<\f35$.
Then it follows from  \eqref{L1beta'} that $|\beta_i^3(t)|<\f45$ for every $t\in]t_{i-1},t_i[,$
and we  define
$$\nu_i(t)\eqdefa\Bigl(-\f{\beta_i^2}{\sqrt{1-|\beta_i^3|^2}},
\f{\beta_i^1}{\sqrt{1-|\beta_i^3|^2}},0\Bigr)(t),\andf
\tau_i(t)\eqdefa\nu_i(t)\times\beta(t),\quad\forall\ t\in[t_{i-1},t_i[.$$
Then for any $t\in]t_{i-1},t_i[$, it is easy to verify that
$\tau_i'(t)$ and $\nu_i'(t)$ exist  whenever $\beta'(t)$ exists
and satisfy $|\tau_i'(t)|+|\nu_i'(t)|\leq C|\beta_i'(t)|$.
Furthermore, it is easy to observe that
 $(\tau_i,\nu_i,\beta_i)$ satisfies \eqref{TNB} in $[t_{i-1},t_i[$.
As a result, by gluing $\tau_i$ (resp. $\nu_i$) together,
we get the desired vector $\tau$ (resp. $\nu$).
This completes the proof of this lemma.
\end{proof}

Thanks to Lemma \ref{lemTNB}, one has $\xi\times\beta
=\tau(\xi\cdot\nu)-\nu(\xi\cdot\tau)$. Then the anisotropic
Sobolev space $\dH^{s_1,s_2}_{\beta(t)}$ given by Definition \ref{XY}
can be equivalently reformulated as
$$\|a\|^2_{\dH^{s_1,s_2}_{\beta(t)}}\eqdefa\int_{\R^3}
\bigl(|\xi\cdot\tau(t)|^2+|\xi\cdot\nu(t)|^2\bigr)^{s_1}
|\xi\cdot\beta(t)|^{2s_2}|\wh a(\xi)|^2\,d\xi.$$
\smallskip

Next, we recall the following anisotropic dyadic operators:
\beq \label{dyadic} \begin{split}
&\Delta_k^{\beta^\perp}a\eqdef\cF^{-1}
\bigl(\varphi(2^{-k}|\xi\times\beta(t)|)\,\widehat{a}(\xi)\bigr),
\quad\Delta_\ell^{\beta}a \eqdef\cF^{-1}
 \bigl(\varphi(2^{-\ell}|\xi\cdot\beta(t)|)\,\widehat{a}(\xi)\bigr),\\
 &S_k^{\beta^\perp}a\eqdef\cF^{-1}
\bigl(\chi(2^{-k}|\xi\times\beta(t)|)\,\widehat{a}(\xi)\bigr),
\quad S_\ell^{\beta}a \eqdef\cF^{-1}
 \bigl(\chi(2^{-\ell}|\xi\cdot\beta(t)|)\,\widehat{a}(\xi)\bigr),
 \end{split}
 \eeq
where $\varphi, \chi: \R\rightarrow\R$
is a smooth function such that
\begin{align*}
&\Supp \varphi \subset \Bigl\{r \in \R\,: \, \frac34 \leq
r\leq \frac83 \Bigr\}\,,\quad\mbox{and}\quad \forall\
 r>0\,:\ \sum_{j\in\Z}\varphi(2^{-j}r)=1,\\
 &\Supp \chi \subset \Bigl\{r \in \R\,: \,
0\leq r\leq \frac43 \Bigr\}\,,\quad\mbox{and}\quad \forall\
 r\geq 0\,:\ \chi(r)+\sum_{j=0}^\infty\varphi(2^{-j}r)=1.
 \end{align*}

\begin{rmk}
One can check for instance \cite{BCD}
for the classical dyadic operators.
The only difference between the case here and the classical one
is that the operators defined by \eqref{dyadic} are anisotropic and  vary with time.
In the following, all the  literatures we cite involve only time-independent
functional spaces.
However, it is easy to verify that,
for any fixed time $t$, the time-dependent spaces
used in this paper share the same properties as the time-independent ones.
\end{rmk}

\begin{defi}\label{anibesov}
{\sl Let $p,\,q_1,\,q_2\in[1,\infty]$
and $s_1,\,s_2\in \R$. $(\dB^{s_1}_{p,q_1})_{\beta^\perp}(\dB^{s_2}_{p,q_2})_{\beta}$
denotes the anisotropic Besov space that
consists of $a\in{\mathcal S}'(\R^3)$ with
$\lim\limits_{j\rightarrow-\infty}
\|(S^{\beta^\perp}_j a,S^{\beta}_j a)\|_{L^\infty}=0$ such that
\begin{equation}\label{defBesov}
\|a\|_{(\dB^{s_1}_{p,q_1})_{\beta^\perp}
(\dB^{s_2}_{p,q_2})_{\beta}}
\eqdef\Bigl\|\Bigl(2^{ks_1}
\bigl\|\bigl(2^{\ell s_2}\|\Delta^{\beta^\perp}_k
\Delta_\ell^{\beta}a\|_{L^p(\R^3)}\bigr)
_{\ell\in\Z}\bigr\|_{\ell^{q_2}(\Z)}\Bigr)
_{k\in\Z}\Bigr\|_{\ell^{q_1}(\Z)}<\infty.
\end{equation}
}\end{defi}

We mention that the order of summation in \eqref{defBesov} is very important.
And we have the following Littlewood-Paley characterization of
the anisotropic Sobolev spaces:
\begin{equation}\label{LPsobolev}
\|a\|_{\dH^{s_1,s_2}_{\beta}}\thicksim
\|a\|_{(\dB^{s_1}_{2,2})_{\beta^\perp}(\dB^{s_2}_{2,2})_{\beta}}
,\quad\bigl(a,b\bigr)_{\dH^{s_1,s_2}_{\beta}}
\thicksim\sum_{k,\ell\in\Z}
2^{2ks_1}2^{2\ell s_2}\bigl(\dhk\dvl a,\dhk\dvl b\bigr)_{L^2}.
\end{equation}

In view of \eqref{defBesov} and \eqref{LPsobolev},
we get, by using H\"older's inequality, that

\begin{lem}\label{leminner}
{\sl For any $s_1,\,s_2\in\R$, and $q_1,\,q_2\in[1,\infty]$
with $q_1',\,q_2'$ being their conjugate numbers, we have
$$ \bigl|(a,b)_{L^2}\bigr|\leq\|a\|_{(\dB^{s_1}_{2,q_1})_{\beta^\perp}
(\dB^{s_2}_{2,q_2})_{\beta}}
\|b\|_{(\dB^{-s_1}_{2,q_1'})_{\beta^\perp}
(\dB^{-s_2}_{2,q_2'})_{\beta}}.$$
}\end{lem}

We also need the folowing anisotropic Bernstein inequalities from
\cite{CZ07, Pa02}:
\begin{lem}\label{lemBern}
{\sl Let $\cB_{\beta^\perp}$ (resp. $\cB_{\beta}$) be a ball of $\R^2_{\beta^\perp}$
(resp. $\R_{\beta}$),
and $\cC_{\beta^\perp}$ (resp. $\cC_{\beta}$) be a ring of $\R^2_{\beta^\perp}$
(resp. $\R_{\beta}$).
Let $1\leq p_2\leq p_1\leq\infty$ and $1\leq q_2\leq q_1\leq \infty$.
Then there hold:
\begin{align*}
&\text{if }~ \Supp \wh a\subset 2^k\cB_{\beta^\perp},~\text{then }~
\|\nablah^N a\|_{L^{p_1}_{\beta^\perp}(L^{q_1}_{\beta})}
\lesssim 2^{k\left(N+2\left(\f1{p_2}-\f1{p_1}\right)\right)}
\|a\|_{L^{p_2}_{\beta^\perp}(L^{q_1}_{\beta})};\\
&\text{if }~ \Supp\wh a\subset 2^\ell\cB_{\beta},~\text{then }~
\|\pa_{\beta}^N a\|_{L^{p_1}_{\beta^\perp}(L^{q_1}_{\beta})}
\lesssim 2^{\ell\left(N+\left(\f1{q_2}-\f1{q_1}\right)\right)}
\|a\|_{L^{p_1}_{\beta^\perp}(L^{q_2}_{\beta})};\\
&\text{if }~ \Supp\wh a\subset 2^k\cC_{\beta^\perp},~\text{then }~
\|a\|_{L^{p_1}_{\beta^\perp}(L^{q_1}_{\beta})}
\lesssim2^{-kN}\|\nablah^N a\|_{L^{p_1}_{\beta^\perp}(L^{q_1}_{\beta})};\\
&\text{if }~ \Supp\wh a\subset2^\ell\cC_{\beta},~\text{then }~
\|a\|_{L^{p_1}_{\beta^\perp}(L^{q_1}_{\beta})}
\lesssim 2^{-\ell N}\|\pa_{\beta}^N a\|_{L^{p_1}_{\beta^\perp}(L^{q_1}_{\beta})},
\end{align*}
where $\pa_{\beta}\eqdefa\beta\cdot\nabla$, and
$\nablah\eqdefa\nabla-\beta(\beta\cdot\nabla)
=\tau(\tau\cdot\nabla)+\nu(\nu\cdot\nabla)$
is the gradient in $\R^2_{\beta^\perp}$.
}\end{lem}

In the following Lemmas \ref{leminter} and \ref{lemproduct},
we shall prove two useful inequalities.
It is worth mentioning that
the precise size of the constants in these inequalities will be crucial
in our proof of Theorem \ref{thm1}.

\begin{lem}\label{leminter}
{\sl For any $\eta\in[0,\f12[$ and $\s\in[0,\f14-\f\eta2[$, there holds
$$\|f\|_{(\dB^{1-\s}_{2,2})_{\beta^\perp}(\dB^{\f12-\eta}_{2,1})_\beta}
\lesssim \Bigl(\f12-2\s-\eta\Bigr)^{-1}\|\pb f\|_{L^2}^{2\s}
\|f\|_{\dH^{\f{3-6\s-2\eta}{2(1-2\s)}}}^{1-2\s}.$$
}\end{lem}

\begin{proof}
By definition \ref{anibesov} and Lemma \ref{lemBern}, we have
\begin{equation}\begin{split}\label{3.6}
\|f&\|_{(\dB^{1-\s}_{2,2})_{\beta^\perp}(\dB^{\f12-\eta}_{2,1})_\beta}^2
=\sum_{k\in\Z}2^{2k(1-\s)}\Bigl(\sum_{\ell\in\Z}
2^{(\f12-\eta)\ell}\|\dhk\dvl f\|_{L^2}\Bigr)^2\\
&\lesssim\sum_{k\in\Z}2^{2k(1-\s)}
\Bigl(2^{2\ell}\sum_{\ell\in\Z}\|\dhk\dvl f\|_{L^2}^2\Bigr)^{2\s}
\Bigl(\sum_{\ell\in\Z}2^{\f{\f12-2\s-\eta}{1-\s}\ell}
\|\dhk\dvl f\|_{L^2}^{\f{1-2\s}{1-\s}}\Bigr)^{2(1-\s)}\\
&\lesssim\Bigl(\sum_{k,\ell\in\Z}2^{2\ell}
\|\dhk\dvl f\|_{L^2}^2\Bigr)^{2\s}\cA\\
&\lesssim\|\pb f\|_{L^2}^{4\s}\cA,
\end{split}\end{equation}
where
$$\cA\eqdef\Bigl\{\sum_{k\in\Z}2^{2k\f{1-\s}{1-2\s}}
\Bigl(\sum_{\ell\in\Z}2^{\f{\f12-2\s-\eta}{1-\s}\ell}
\|\dhk\dvl f\|_{L^2}^{\f{1-2\s}{1-\s}}\Bigr)^{\f{2(1-\s)}{1-2\s}}\Bigr\}^{1-2\s}.$$

By using the elementary inequality: $|a+b|^s\leq2^s(|a|^s+|b|^s),~\forall~s>0$,
we deduce
$$\cA\leq 2^{2(1-\s)}(\cA_1+\cA_2)\leq4(\cA_1+\cA_2) \with$$
\begin{align*}
&\cA_1\eqdef\Bigl\{\sum_{k\in\Z}2^{2k\f{1-\s}{1-2\s}}
\Bigl(\sum_{\ell\leq k}2^{\f{\f12-2\s-\eta}{1-\s}\ell}
\|\dhk\dvl f\|_{L^2}^{\f{1-2\s}{1-\s}}\Bigr)^{\f{2(1-\s)}{1-2\s}}\Bigr\}^{1-2\s},\\
&\cA_2\eqdef\Bigl\{\sum_{k\in\Z}2^{2k\f{1-\s}{1-2\s}}
\Bigl(\sum_{\ell>k}2^{\f{\f12-2\s-\eta}{1-\s}\ell}
\|\dhk\dvl f\|_{L^2}^{\f{1-2\s}{1-\s}}\Bigr)^{\f{2(1-\s)}{1-2\s}}\Bigr\}^{1-2\s}.
\end{align*}
Noticing that $\f{\f12-2\s-\eta}{1-\s}>0$,
and the operator $\dvl$ is $L^2$ bounded, we infer
\begin{equation}\begin{split}\label{3.7}
\cA_1&\lesssim\Bigl\{\sum_{k\in\Z}2^{2k\f{1-\s}{1-2\s}}\|\dhk f\|_{L^2}^2
\Bigl(\sum_{\ell\leq k}2^{\f{\f12-2\s-\eta}{1-\s}\ell}
\Bigr)^{\f{2(1-\s)}{1-2\s}}\Bigr\}^{1-2\s}\\
&\lesssim\Bigl(\f{1-\s}{\f12-2\s-\eta}\Bigr)^{2(1-\s)}
\Bigl(\sum_{k\in\Z}2^{k\f{3-6\s-2\eta}{1-2\s}}\|\dhk f\|_{L^2}^2\Bigr)^{1-2\s}.
\end{split}\end{equation}
While by using Plancherel's identity and \eqref{embeta}, we obtain
\begin{align*}
\sum_{k\in\Z}2^{k\f{3-6\s-2\eta}{1-2\s}}\|\dhk f\|_{L^2}^2
&=\sum_{k,\ell\in\Z}2^{k\f{3-6\s-2\eta}{1-2\s}}
\|\dhk\dvl f\|_{L^2}^2\\
&=\|f\|_{\dH^{\f{3-6\s-2\eta}{2(1-2\s)},0}}^2
\lesssim\|f\|_{\dH^{\f{3-6\s-2\eta}{2(1-2\s)}}}^2.
\end{align*}
By inserting the above estimate into \eqref{3.7}
and using $1-\s\in]\f34,1]$, we achieve
\begin{equation}\label{3.8}
\cA_1\lesssim \Bigl(\f12-2\s-\eta\Bigr)^{-2}
\|f\|_{\dH^{\f{3-6\s-2\eta}{2(1-2\s)}}}^{2(1-2\s)}.
\end{equation}

On the other hand, we have
\begin{equation}\begin{split}\label{3.9}
\cA_2&=\Bigl\{\sum_{k\in\Z}
\Bigl(\sum_{\ell>k}2^{\f{3-6\s-2\eta}{2(1-\s)}\ell}
\|\dhk\dvl f\|_{L^2}^{\f{1-2\s}{1-\s}}2^{k-\ell}
\Bigr)^{\f{2(1-\s)}{1-2\s}}\Bigr\}^{1-2\s}\\
&\lesssim\Bigl\{\sum_{k\in\Z}
\Bigl(\sum_{\ell>k}2^{\f{3-6\s-2\eta}{1-2\s}\ell}
\|\dhk\dvl f\|_{L^2}^2\Bigr)\Bigl(\sum_{\ell>k}2^{2(k-\ell)(1-\s)}
\Bigr)^{\f{1}{1-2\s}}\Bigr\}^{1-2\s}\\
&\lesssim\Bigl(\sum_{k,\ell\in\Z}
2^{\f{3-6\s-2\eta}{1-2\s}\ell}
\|\dhk\dvl f\|_{L^2}^2\Bigr)^{1-2\s}\\
&=\|f\|_{\dH^{0,\f{3-6\s-2\eta}{2(1-2\s)}}_\beta}^{2(1-2\s)}
\leq\|f\|_{\dH^{\f{3-6\s-2\eta}{2(1-2\s)}}}^{2(1-2\s)}.
\end{split}\end{equation}

By substituting \eqref{3.8} and \eqref{3.9} into \eqref{3.6},
we complete the proof of this lemma.
\end{proof}

Before proceeding, we recall Bony's
decomposition in the $\R^2_{\beta^\perp}$ variables from \cite{Bo81}:
\beq\label{bony}
\begin{split}
&ab=T^{\beta^\perp}_ab+T^{\beta^\perp}_ba+R^{\beta^\perp}(a,b)\with
T^{\beta^\perp}_ab\eqdef\sum_{k\in\Z}S^{\beta^\perp}_{k-1}a\,\D^{\beta^\perp}_k b,
\\
& R^{\beta^\perp}(a,b)
\eqdef\sum_{k\in\Z}\D^{\beta^\perp}_k a\,\wt{\D}^{\beta^\perp}_{k}b \andf
\wt{\Delta}^{\beta^\perp}_{k}
\eqdef\D^{\beta^\perp}_{k-1}+\D^{\beta^\perp}_k+\D^{\beta^\perp}_{k+1}.
\end{split}\eeq
And  Bony's decomposition in the $\R_{\beta}$ variable can be defined in the same way.

By applying  Bony's decompositions \eqref{bony}, we shall prove
the following law of product:

\begin{lem}\label{lemproduct}
{\sl For any $s_1,s_2<1$ with $s_1+s_2>0$, and
any $r_1,r_2<\f12$ with $r_1+r_2\geq0$, we use
$C_{s_1,s_2}$ and $C_{r_1,r_2}$ to denote $\max\bigl\{(1-s_1)^{-\f12},
(1-s_2)^{-\f12},(s_1+s_2)^{-\f12}\bigr\}$ and $\max\bigl\{(1-2r_1)^{-\f12},
(1-2r_2)^{-\f12}\bigr\}$ respectively.
Then we have
$$\|fg\|_{(\dB^{s_1+s_2-1}_{2,2})_{\beta^\perp}
(\dB^{r_1+r_2-\f12}_{2,\infty})_\beta}
\lesssim C_{s_1,s_2}C_{r_1,r_2}
\|f\|_{\dH^{s_1,r_1}_\beta}
\|g\|_{\dH^{s_2,r_2}_\beta}.$$
}\end{lem}

\begin{proof}
{\bf Step 1.} Let us first show that for any smooth  functions $a$ and $b$, there holds
\begin{equation}\label{product1}
\sup_{\ell\in\Z}
2^{\left(r_1+r_2-\f12\right)\ell}\|\dvl(ab)\|_{L^2_\beta}
\lesssim C_{r_1,r_2}\|a\|_{\dH^{r_1}_\beta}\|b\|_{\dH^{r_2}_\beta}.
\end{equation}

Indeed for any $\ell\in\Z$, we get, by using Bony's decomposition \eqref{bony} in the $\R_{\beta}$ variable
and Lemma \ref{lemBern}, that
$$\|\dvl(ab)\|_{L^2_\beta}
\leq \|\dvl T^\beta_ab\|_{L^2_\beta}
+\|\dvl T^\beta_ba\|_{L^2_\beta}+2^{\f\ell2}\|\dvl R^\beta(ab)\|_{L^1_\beta}.$$
For the para-product part, we have
\begin{align*}
2^{\left(r_1+r_2-\f12\right)\ell}\|\dvl T^\beta_ab\|_{L^2_\beta}
&\leq2^{\left(r_1+r_2-\f12\right)\ell}\sum_{|\ell'-\ell|\leq4}
\|\Svlp a\|_{L^\infty_\beta}\|\dvlp b\|_{L^2_\beta}\\
&\lesssim\sup_{\ell'\in\Z}2^{\left(r_1+r_2-\f12\right)\ell'}
\|\Svlp a\|_{L^\infty_\beta}\|\dvlp b\|_{L^2_\beta}.
\end{align*}
While it follows from  Lemma \ref{lemBern} and the fact: $r_1<\f12,$ that
\begin{align*}
\|\Svlp a\|_{L^\infty_\beta}
\lesssim\sum_{\ell''\leq\ell'-2}2^{\f{\ell''}2}\|\dvlpp a\|_{L^2_\beta}
&\lesssim
\bigl(\sum_{\ell''\leq\ell'-2}2^{2r_1\ell''}\|\dvlpp a\|_{L^2_\beta}^2\bigr)^{\f12}
\bigl(\sum_{\ell''\leq\ell'-2}2^{(1-2 r_1)\ell''}\bigr)^{\f12}\\
&\lesssim(1-2r_1)^{-\f12} 2^{(\f12-r_1)\ell'}\|a\|_{\dH^{r_1}_\beta}.
\end{align*}
As a result,  for any $\ell\in\Z,$  we have
\begin{equation}\begin{split}\label{3.11}
2^{(r_1+r_2-\f12)\ell}\|\dvl T^\beta_ab\|_{L^2_\beta}
&\lesssim(1-2r_1)^{-\f12}\|a\|_{\dH^{r_1}_\beta}
\sup_{\ell'\in\Z}2^{r_2\ell'}\|\dvlp b\|_{L^2_\beta}\\
&\lesssim(1-2r_1)^{-\f12}
\|a\|_{\dH^{r_1}_\beta}\|b\|_{\dH^{r_2}_\beta}.
\end{split}\end{equation}
Exactly along the same line, we infer
\begin{equation}\label{3.12}
2^{(r_1+r_2-\f12)\ell}\|\dvl T^\beta_ab\|_{L^2_\beta}
\lesssim(1-2r_2)^{-\f12}
\|a\|_{\dH^{r_1}_\beta}\|b\|_{\dH^{r_2}_\beta}.
\end{equation}

Next, for the remainder term, we get, by first taking summation in $\ell$
and then using h\"older's inequality
as well as the fact: $r_1+r_2\geq0,$ that
\begin{equation}\begin{split}\label{3.13}
2^{(r_1+r_2)\ell}\|\dvl R^\beta(ab)\|_{L^1_\beta}
&\lesssim\sum_{\ell'\geq \ell-3}2^{(r_1+r_2)\ell'}\|\dvlp a\|_{L^2_\beta}
\|\wtdvlp b\|_{L^2_\beta}\\
&\lesssim\Bigl(\sum_{\ell'\in\Z}2^{2r_1\ell'}
\|\dvlp a\|_{L^2_\beta}^2\Bigr)^{\f12}
\Bigl(\sum_{\ell'\in\Z}2^{2r_2\ell'}\|\dvlp b\|_{L^2_\beta}^2\Bigr)^{\f12}\\
&\lesssim\|a\|_{\dH^{r_1}_\beta}\|b\|_{\dH^{r_2}_\beta}.
\end{split}\end{equation}

By combining the estimates \eqref{3.11}-\eqref{3.13}, we obtain
 \eqref{product1}.

\noindent{\bf Step 2.} By using Bony's decomposition in the $\R^2_{\beta^\perp}$ variables, we have
\begin{equation}\begin{split}\label{3.14}
\|fg\|_{(\dB^{s_1+s_2-1}_{2,2})_{\beta^\perp}
(\dB^{r_1+r_2-\f12}_{2,\infty})_\beta}^2
=&\sum_{k\in\Z}2^{2 k(s_1+s_2-1)}
\sup_{\ell\in\Z}2^{2\ell(r_1+r_2-\f12)}
\|\dhk\dvl (f g)\|_{L^2}^2\\
\lesssim&\sum_{k\in\Z}2^{2 k(s_1+s_2-1)}
\sup_{\ell\in\Z}2^{2\ell(r_1+r_2-\f12)}
\Bigl(\|\dhk\dvl T^{\beta^\perp}_f g\|_{L^2}^2\\
+&\|\dhk\dvl T^{\beta^\perp}_g f\|_{L^2}^2
+2^{2k}\|\dhk\dvl R^{\beta^\perp}(f,g)\|_{L^1_{\beta^\perp}(L^2_\beta)}^2\Bigr).
\end{split}\end{equation}
For any fixed $k$, we get, by using
\eqref{product1}, that
\begin{align*}
&\sup_{\ell\in\Z}
2^{2\ell(r_1+r_2-\f12)}\Bigl(\|\dhk\dvl T^{\beta^\perp}_f g\|_{L^2}^2
+\|\dhk\dvl T^{\beta^\perp}_g f\|_{L^2}^2\Bigr)\\
&=\sum_{|k'-k|\leq4}\int_{\R^2_{\beta^\perp}}\sup_{\ell\in\Z}
2^{2\ell(r_1+r_2-\f12)}\Bigl(\|\dvl (\Shkp f\dhkp g)\|_{L^2_\beta}^2
+\|\dvl(\Shkp g\dhkp f)\|_{L^2_\beta}^2\Bigr)\,dx_{\beta^\perp}\\
&\lesssim C_{r_1,r_2}^2\sum_{|k'-k|\leq4}
\Bigl(\|\Shkp f\|_{L^\infty_{\beta^\perp}(\dH^{r_1}_\beta)}^2
\|\dhkp g\|_{L^2_{\beta^\perp}(\dH^{r_2}_\beta)}^2
+\|\dhkp f\|_{L^2_{\beta^\perp}(\dH^{r_1}_\beta)}^2
\|\Shkp g\|_{L^\infty_{\beta^\perp}(\dH^{r_2}_\beta)}^2\Bigr).
\end{align*}
By multiplying the above inequality by $2^{2 k\left(s_1+s_2-1\right)}$
and then summing up the resulting inequalities for  $k\in\Z,$ we find
\begin{align*}
&\sum_{k\in\Z}2^{2 k(s_1+s_2-1)}\sup_{\ell\in\Z}
2^{2\ell(r_1+r_2-\f12)}\Bigl(\|\dhk\dvl T^{\beta^\perp}_f g\|_{L^2}^2
+\|\dhk\dvl T^{\beta^\perp}_g f\|_{L^2}^2\Bigr)\\
&\lesssim C_{r_1,r_2}^2\sum_{k\in\Z}2^{2 k(s_1+s_2-1)}
\Bigl(\|\Shk f\|_{L^\infty_{\beta^\perp}(\dH^{r_1}_\beta)}^2
\|\dhk g\|_{L^2_{\beta^\perp}(\dH^{r_2}_\beta)}^2\\
&\qquad\qquad\qquad\qquad\qquad\qquad+\|\dhk f\|_{L^2_{\beta^\perp}(\dH^{r_1}_\beta)}^2
\|\Shk g\|_{L^\infty_{\beta^\perp}(\dH^{r_2}_\beta)}^2\Bigr)\\
&\lesssim C_{r_1,r_2}^2\sum_{k\in\Z}2^{2 k(s_1+s_2-1)}
\Bigl(\sum_{k'\leq k-2}
2^{k'}\|\dhkp f\|_{L^2_{\beta^\perp}(\dH^{r_1}_\beta)}\Bigr)^2
\|\dhk g\|_{L^2_{\beta^\perp}(\dH^{r_2}_\beta)}^2\\
&\qquad+C_{r_1,r_2}^2\sum_{k\in\Z}2^{2 k(s_1+s_2-1)}
\|\dhk f\|_{L^2_{\beta^\perp}(\dH^{r_1}_\beta)}^2
\Bigl(\sum_{k'\leq k-2}
2^{k'}\|\dhkp g\|_{L^2_{\beta^\perp}(\dH^{r_2}_\beta)}\Bigr)^2.
\end{align*}
By applying H\"older's inequality and using the fact that $s_1<1$, we deduce for any $k\in\Z$ that
\begin{align*}
\Bigl(\sum_{k'\leq k-2}2^{k'}
\|\dhkp f\|_{L^2_{\beta^\perp}(\dH^{r_1}_\beta)}\Bigr)^2
&\lesssim
\Bigl(\sum_{k'\leq k-2}2^{2(1-s_1)k'}\Bigr)
\Bigl(\sum_{k'\leq k-2}
2^{2s_1k'}\|\dhkp f\|_{L^2_{\beta^\perp}(\dH^{r_1}_\beta)}^2\Bigr)\\
&\lesssim(1-s_1)^{-1} 2^{2k(1-s_1)}\|f\|_{\dH^{s_1,r_1}}^2.
\end{align*}
Similar estimate holds for $\Bigl(\sum\limits_{k'\leq k-2}
2^{k'}\|\dhkp g\|_{L^2_{\beta^\perp}(\dH^{r_2}_\beta)}\Bigr)^2$.
As a result, we obtain
\begin{equation}\begin{split}\label{3.15}
\sum_{k\in\Z}2^{2 k(s_1+s_2-1)}&\sup_{\ell\in\Z}
2^{2\ell(r_1+r_2-\f12)}\Bigl(\|\dhk\dvl T^{\beta^\perp}_f g\|_{L^2}^2
+\|\dhk\dvl T^{\beta^\perp}_g f\|_{L^2}^2\Bigr)\\
&\lesssim\max\bigl\{(1-s_1)^{-1},(1-s_2)^{-1}\bigr\}C_{r_1,r_2}^2
\|f\|_{\dH^{s_1,r_1}}^2\|g\|_{\dH^{s_2,r_2}}^2.
\end{split}\end{equation}

While for the remainder term,  we get, by using \eqref{product1} once again, that
\begin{equation}\begin{split}\label{3.16}
&\sum_{k\in\Z}2^{2 k(s_1+s_2)}\sup_{\ell\in\Z}
2^{2\ell(r_1+r_2-\f12)}
\|\dhk\dvl R^{\beta^\perp}(f,g)\|_{L^1_{\beta^\perp}(L^2_\beta)}^2\\
&\lesssim\sum_{k\in\Z}2^{2 k(s_1+s_2)}
\sup_{\ell\in\Z}
\Bigl(\sum_{k'\geq k-3}2^{\ell(r_1+r_2-\f12)}\|\dvl(\dhkp f\wtdhkp g)\|_{L^1_{\beta^\perp}(L^2_\beta)}\Bigr)^2\\
&\lesssim\sum_{k\in\Z}2^{2 k(s_1+s_2)}
\Bigl(\sum_{k'\geq k-3}C_{r_1,r_2}
\|\dhkp f\|_{L^2_{\beta^\perp}(\dH^{r_1}_\beta)}
\|\wtdhkp g\|_{L^2_{\beta^\perp}(\dH^{r_2}_\beta)}\Bigr)^2\\
&\lesssim C_{r_1,r_2}^2\sum_{k\in\Z}
\Bigl(\sum_{k'\geq k-3}2^{(k-k')(s_1+s_2)}
2^{k' s_1}\|\dhkp f\|_{L^2_{\beta^\perp}(\dH^{r_1}_\beta)}
2^{k' s_2}\|\wtdhkp g\|_{L^2_{\beta^\perp}(\dH^{r_2}_\beta)}\Bigr)^2.
\end{split}\end{equation}
By using Young's inequality, we obtain
\begin{align*}
&\sum_{k\in\Z}
\Bigl(\sum_{k'\geq k-3}2^{(k-k')(s_1+s_2)}
2^{k' s_1}\|\dhkp f\|_{L^2_{\beta^\perp}(\dH^{r_1}_\beta)}
2^{k' s_2}\|\wtdhkp g\|_{L^2_{\beta^\perp}(\dH^{r_2}_\beta)}\Bigr)^2\\
&\lesssim\Bigl(\sum_{k'\in\Z}2^{k' s_1}
\|\dhkp f\|_{L^2_{\beta^\perp}(\dH^{r_1}_\beta)}
2^{k' s_2}\|\wtdhkp g\|_{L^2_{\beta^\perp}(\dH^{r_2}_\beta)}\Bigr)^2
\Bigl(\sum_{k'\leq3}2^{2k'(s_1+s_2)}\Bigr)\\
&\lesssim(s_1+s_2)^{-1}\|f\|_{\dH^{s_1,r_1}}^2\|g\|_{\dH^{s_2,r_2}}^2.
\end{align*}
Inserting the above estimate into \eqref{3.16} yields
\begin{equation}\begin{split}\label{3.17}
\sum_{k\in\Z}2^{2 k(s_1+s_2)}
\sup_{\ell\in\Z}2^{2\ell(r_1+r_2-\f12)}
\|&\dhk\dvl R^{\beta^\perp}(f,g)\|_{L^1_{\beta^\perp}(L^2_\beta)}^2\\
&\lesssim(s_1+s_2)^{-1} C_{r_1,r_2}^2\|f\|_{\dH^{s_1,r_1}}^2\|g\|_{\dH^{s_2,r_2}}^2.
\end{split}\end{equation}

By substituting \eqref{3.15} and \eqref{3.17}
into \eqref{3.14}, we complete the proof of Lemma \ref{lemproduct}.
\end{proof}

\section{The proof of Theorem \ref{thm1}}\label{secthm1}

In this section, we shall present the proof of Theorem \ref{thm1}.

Let $\beta$ be given by Theorem \ref{thm1},
and the corresponding $\tau,~\nu\in\Omega(T^\ast)$
by Lemma \ref{lemTNB}.
We denote the jump discontinuity set of
$\tau,~\nu$ and $\beta$ in $]0,T^\ast[$ by
$\left\{T_1,\cdots,T_{n-1} \right\}$. Then it remains  to prove that
a strong solution $u$ to $(NS)$ can be extended beyond $T^\ast$ provided
\begin{equation}\label{4.1}
\sum_{i=1}^n\int_{T_{i-1}}^{T_i}\Bigl(|\tau'(t)|^2
+|\nu'(t)|^2+|\beta'(t)|^2+\|u(t)\cdot \beta(t)
\|_{\dH^{\f32}}^2\Bigr)\,dt<\infty.
\end{equation}
 Here we denote $T_0\eqdefa 0$ and $T_n\eqdefa T^\ast$.

Before preceding, let us introduce the following notations:
$$x_\beta\eqdef \beta\cdot x,\quad
x_{\beta^\perp}\eqdef x_\tau\tau+x_\nu\nu,\quad
u^\beta\eqdef \beta\cdot u,\quad u^{\beta^\perp}\eqdef u^\tau\tau+u^\nu\nu,
\quad\omega^\beta\eqdef\pa_\tau u^\nu-\pa_\nu u^\tau,$$
$$\pa_{\beta}\eqdef\beta\cdot\nabla,\quad
\nablah\eqdef\tau\pa_\tau+\nu\pa_\nu,\quad
\nablah^\perp\eqdef-\tau\pa_\nu+\nu\pa_\tau,\quad
\D_{\beta^\perp}\eqdef\pa_\tau^2+\pa_\nu^2.$$

Due to $\dive u=\pa_\tau u^\tau+\pa_\nu u^\nu+\pa_\beta u^\beta=0$, we have
$$\nablah\cdot\uh=-\pb\ub,\quad\nablah^\perp\cdot\uh=\om.$$
Then we have the following version of
Helmholtz decomposition for $\uh$:
\begin{equation}\label{4.2}
\uh=\ucurl+\udiv,\quad\mbox{with}\quad
\ucurl
\eqdef\nablah^\perp \Dh^{-1} \om
\andf \udiv\eqdef-\nablah\Dh^{-1}\pb\ub.
\end{equation}
As a result, we deduce from the equations of $(NS)$ that $\om$ and $\pb\ub$ verify
\begin{equation}\label{4.3}
\left\{
\begin{array}{l}
\pa_t\om-\tau'\cdot(\nabla u^\nu-\pa_\nu u)
-\nu'\cdot(\pa_\tau u-\nabla u^\tau)
+u\cdot\nabla\om-\D\om\\
\qquad\qquad
=\pb\ub\om-\pb\uh\cdot\nablah\ub,\\
\pa_t \pb\ub-\beta'\cdot(\pa_\beta u+\nabla\ub)
+u\cdot\nabla\pb\ub-\D\pb\ub+\pb u\cdot\nabla \ub=-\pb^2P.
\end{array}
\right.
\end{equation}

On the other hand, it follows from the rotational symmetry that
$$u\cdot\nabla=u^\tau\pa_\tau
+u^\nu\pa_\nu+u^\beta\pa_\beta,\andf
\D=\pa_\tau^2+\pa_\nu^2+\pa_\beta^2.$$
So that we represent the pressure function $P$ as
\begin{equation}\label{4.4}
P=-\D^{-1}\Bigl(\sum\limits_{\ell,m\in\{\tau,\nu,\beta\}}
\pa_\ell u^m\pa_m u^\ell\Bigr).
\end{equation}

Let us first focus on the estimate of the solution to \eqref{4.3} on the first time interval $[0,T_1]$.
The following {\it a priori} estimate to \eqref{4.3} will
play a key role in our proof of Theorem \ref{thm1}, whose proof will be
postponed to Section \ref{secprop1}.

\begin{prop}\label{prop1}
{\sl Let $u$  be a strong solution to $(NS)$ on $[0,T[$ for some $T>0$.
Then for any $t\in[0,T[$ and any $\s\in\left]0, 1/5\right],$  there holds
\begin{equation}\begin{split}\label{ineqprop1}
\f{d}{dt}&\bigl(\|\om\|_{L^2}^2+\|\pb\ub\|_{L^2}^2\bigr)
+\|\nabla\om\|_{L^2}^2+\|\nabla\pb\ub\|_{L^2}^2
\leq C\bigl(\|\om\|_{L^2}^2+\|\pb\ub\|_{L^2}^2\bigr)\|\ub\|_{\dH^{\f32}}^2\\
&+C\|u_0\|_{L^2}^2\bigl(|\tau'|^2+|\nu'|^2+|\beta'|^2\bigr)
+\f{C}{\s}\bigl(\|\om\|_{L^2}^{\f{2}{1-\s}}+\|\pb\ub\|_{L^2}^{\f{2}{1-\s}}\bigr)
\|\ub\|_{\dH^{\f32}}^{\f{2(1-2\s)}{1-\s}}\|\nabla u\|_{L^2}^{\f{2\s}{1-\s}}.
\end{split}\end{equation}
}\end{prop}

Now let us denote $F(t)\eqdef\|\om(t)\|_{L^2}^2+\|\pb\ub(t)\|_{L^2}^2$.
Then \eqref{ineqprop1} implies
\begin{align*}
\f{d}{dt} F(t)\leq& C\f{1-\s}{\s}F(t)^{1+\f{\s}{1-\s}}
\|\ub\|_{\dH^{\f32}}^{\f{2(1-2\s)}{1-\s}}\|\nabla u\|_{L^2}^{\f{2\s}{1-\s}}\\
&+C\|\ub(t)\|_{\dH^{\f32}}^2F(t)
+C\|u_0\|_{L^2}^2\bigl(|\tau'|^2+|\nu'|^2+|\beta'|^2\bigr),
\quad\forall\ \s\in]0, 1/5].
\end{align*}
In view of \eqref{4.1}, and the energy inequality for the solution $u$ of $(NS):$
\begin{equation}\label{energyineq}
\|u\|_{L^\infty_t(L^2)}^2+2\|\nabla u\|_{L^2_t(L^2)}\leq\|u_0\|_{L^2}^2,
\quad\forall\ t>0,
\end{equation}
we deduce for any $t<T_1$ that
$$\int_0^t\|\ub(s)\|_{\dH^{\f32}}^{\f{2(1-2\s)}{1-\s}}
\|\nabla u(s)\|_{L^2}^{\f{2\s}{1-\s}}\,ds
\leq\|\ub\|_{L^2_t(\dH^{\f32})}^{\f{2(1-2\s)}{1-\s}}
\|\nabla u\|_{L^2_t(L^2)}^{\f{2\s}{1-\s}}<\infty.$$
Hence we get, by using Corollary \ref{cor2.1}, that
\begin{equation}\label{4.7}
F(t)=\|\om(t)\|_{L^2}^2+\|\pb\ub(t)\|_{L^2}^2\leq \wt L_1
\end{equation}
for some positive constant $\wt L_1$ depending only on
$$\|u_0\|_{H^1},\quad\|(\tau',\nu',\beta')\|_{L^2([0,T_1[)},
\andf \|\ub\|_{L^2_{T_1}(\dH^{\f32})}.$$
By using the Helmholtz decomposition \eqref{4.2}, and
substituting the estimate \eqref{4.7} into the right-hand side of \eqref{ineqprop1},
and then integrating the resulting inequality over $[0,t],$
we achieve
\begin{equation}\begin{split}\label{4.8}
\|\nablah\uh\|_{L^\infty_t(L^2)\cap L^2_t(\dH^1)}^2
&\leq C\|(\om,\pb\ub)\|_{L^\infty_t(L^2)\cap L^2_t(\dH^1)}^2\\
&\leq C\wt L_1\|\ub\|_{L^2_{T_1}(\dH^{\f32})}^2
+C\|u_0\|_{L^2}^2\|(\tau',\nu',\beta')\|_{L^2([0,T_1[)}^2\\
&\quad+\f{C}{\s}\wt L_1^{\f{1}{1-\s}}
\|\ub\|_{L^2_{T_1}(\dH^{\f32})}^{\f{2(1-2\s)}{1-\s}}
\|u_0\|_{L^2}^{\f{2\s}{1-\s}}\eqdef\bar L_1,\quad\forall\ t<T_1.
\end{split}\end{equation}

With \eqref{4.8} at hand, we are  now in a position to complete the proof of Theorem \ref{thm1}.

\begin{proof}[Proof of Theorem \ref{thm1}]
By applying the space gradient $\nabla$ to $(NS),$ we find
\begin{equation}\label{6.1}
\pa_t\nabla u+u\cdot\nabla(\nabla u)+\nabla u\cdot\nabla u
-\D\nabla u+\nabla^2 P=0.
\end{equation}

By using integration by parts and the divergence-free condition of $u,$
we obtain
$$\int_{\R^3}\Bigl(u\cdot\nabla(\nabla u)
+\nabla^2 P\Bigr):(\nabla u)\,dx
=-\int_{\R^3}\Bigl(\f12|\nabla u|^2\dive u
+(\nabla\dive u)\cdot\nabla P\Bigr)\,dx=0.$$
So that we get, by taking $L^2$ inner product of \eqref{6.1} with $\nabla u,$ that
\begin{equation}\label{6.2}
\f12\f{d}{dt}\|\nabla u\|_{L^2}^2+\|\nabla^2 u\|_{L^2}^2
=-\int_{\R^3}\cB\,dx,
\with \cB=\sum_{\ell,m\in\{\tau,\nu,\beta\}}
\bigl(\pa_\ell u\cdot\nabla u^m\bigr)\pa_\ell u^m.
\end{equation}

It is crucial to notice that all the terms
in $\cB$ are of the form
$\nabla u\otimes\nabla u\otimes\nablah\uh.$
In fact, this is obvious the case when ${\ell,m\in\{\tau,\nu\}}$.
While when $\ell=m=\beta$, due to $\dive u=0,$ we have
$$\bigl(\pa_\beta u\cdot\nabla \ub\bigr)\pa_\beta \ub
=-\bigl(\pa_\beta u\cdot\nabla \ub\bigr)\bigl(\nablah\cdot\uh\bigr).$$
When $\ell=\beta$ and $m\in\{\tau,\nu\}$, we have
\begin{align*}
\bigl(\pb u\cdot\nabla \uh\bigr)\cdot\pb \uh
&=\bigl(\pb \uh\cdot\nablah \uh+\pb \ub\pb \uh\bigr)\cdot\pb \uh\\
&=\bigl(\pb \uh\cdot\nablah \uh-(\nablah\cdot\uh)\pb \uh\bigr)\cdot\pb \uh.
\end{align*}
When $m=\beta$ and $\ell\in\{\tau,\nu\}$, we have
\begin{align*}
\bigl(\nablah u\cdot\nabla \ub\bigr)\cdot\nablah \ub
&=\bigl(\nablah \uh\cdot\nablah \ub+\nablah \ub\pb \ub\bigr)\cdot\nablah \ub\\
&=\bigl(\nablah \uh\cdot\nablah \ub-\nablah\ub(\nablah\cdot\uh)\bigr)\cdot\nablah \ub.
\end{align*}
As a result, the right-hand side of \eqref{6.2} can be handled as follows
\begin{equation}\begin{split}\label{6.3}
\Bigl|\int_{\R^3}\cB\,dx\Bigr|
&\leq C\|\nabla u\|_{L^2}\|\nabla u\|_{L^6}\|\nablah\uh\|_{L^3}\\
&\leq C\|\nabla u\|_{L^2}\|\nabla^2 u\|_{L^2}
\|\nablah\uh\|_{L^2}^{\f12}\|\nabla\nablah\uh\|_{L^2}^{\f12}\\
&\leq\f12\|\nabla^2 u\|_{L^2}^2+C\|\nabla u\|_{L^2}^2
\|\nablah\uh\|_{L^2}\|\nabla\nablah\uh\|_{L^2}.
\end{split}\end{equation}

By substituting \eqref{6.3} into \eqref{6.2},
and then using Gronwall's inequality together with the estimates
\eqref{energyineq} and \eqref{4.8}, we get for any $t<T_1$ that
\begin{equation}\begin{split}\label{ineqprop2}
\|\nabla u\|_{L^\infty_t(L^2)}^2+\|\nabla^2 u\|_{L^2_t(L^2)}^2
&\leq \|\nabla u_0\|_{L^2}^2\exp\Bigl(C\|\nabla u\|_{L^2_t(L^2)}
\|\nabla\nablah\uh\|_{L^2_t(L^2)}\Bigr)\\
&\leq \|\nabla u_0\|_{L^2}^2\exp\Bigl(C\|u_0\|_{L^2}
\bar L_1^{\f12}\Bigr)\eqdef L_1.
\end{split}\end{equation}

Thanks to \eqref{ineqprop2}, we deduce from the classical
well-posedness theory for the system $(NS)$ in $H^1$ that,
$u$  can be extended to be a strong solution of $(NS)$
at least on $[0,T_1+C L_1^{-2}]$, and there holds
\begin{equation}\label{4.13}
\|\nabla u\|_{L^\infty_t(L^2)}^2+\|\nabla^2 u\|_{L^2_t(L^2)}^2
\leq2L_1,\quad\forall\ t\in[0,T_1+C L_1^{-2}].
\end{equation}

In particular, the estimate \eqref{4.13} ensures that $\|\nabla u(T_1)\|_{L^2}^2\leq2L_1$.
Then we can view $T_1$ as our new initial time, and solve $(NS)$
on $[T_1,T_2[$. Then along the same line to the proof of \eqref{4.13}, we find that $u$ exists
at least on $[0,T_2+C L_2^{-2}]$ with $\|\nabla u(T_2)\|_{L^2}^2\leq2L_2$
for some constant $L_2>0$.

By repeating the above procedure for $n-2$ more times, we conclude
that $u$  can actually be extended beyond the time $T_n=T^\ast$
with lifespan no less than $C L_n^{-2}$ for some constant $L_n>0$.
This completes the proof of Theorem \ref{thm1}.
\end{proof}

\section{The proof of Proposition \ref{prop1}}\label{secprop1}

The aim of this section is to present the proof of Proposition \ref{prop1}.

\begin{proof}[Proof of Proposition \ref{prop1}]
By taking $L^2$ inner product of the first  equation in \eqref{4.3}
with $\om$, we get
\begin{equation}\begin{split}\label{5.1}
\f12\f{d}{dt}\|\om\|_{L^2}^2+&\|\nabla\om\|_{L^2}^2
=\int_{\R^3}\Bigl(\tau'\cdot(\nabla u^\nu-\pa_\nu u)
+\nu'\cdot(\pa_\tau u-\nabla u^\tau)\Bigr)\om\,dx\\
&+\int_{\R^3} \pb\ub |\om|^2\,dx
-\int_{\R^3}\bigl(\pb\uh\cdot\nablah\ub\bigr) \om\,dx
\eqdef I_1+I_2+I_3.
\end{split}\end{equation}

For $I_1$, by using integration by parts,
and the fact: $|\tau|=|\nu|=1$, we get
\begin{equation}\begin{split}\label{5.2}
|I_1|&=\Bigl|\int_{\R^3}\Bigl(\tau'\cdot(u^\nu\nabla\om-u\pa_\nu\om)
+\nu'\cdot(u\pa_\tau\om- u^\tau\nabla\om)\Bigr)\,dx\Bigr|\\
&\leq2\bigl(|\tau'|+|\nu'|\bigr)
\|u\|_{L^2}\|\nabla \om\|_{L^2}\\
&\leq\f18\|\nabla \om\|_{L^2}^2
+C\|u_0\|_{L^2}^2\bigl(|\tau'|^2+|\nu'|^2\bigr),
\end{split}\end{equation}
where in the last step, we used the energy inequality \eqref{energyineq}.

While for $I_2,$  we get, by using Sobolev embedding theorem
and $|\beta|=1,$ that
\begin{equation}\begin{split}\label{5.3}
|I_2|&\leq\|\pb\ub\|_{L^3}\|\om\|_{L^2}\|\om\|_{L^6}\\
&\leq\f18\|\nabla\om\|_{L^2}^2+C\|\om\|_{L^2}^2\|\nabla\ub\|_{\dH^{\f12}}^2.
\end{split}\end{equation}

For the most troublesome term $I_3$,
by using Lemma \ref{leminner}, Lemma \ref{leminter} for $\eta=0$
and Lemma \ref{lemproduct}, for any $\s\in]0, 1/5]$, we deduce  that
\begin{align*}
|I_3|&\leq\|\nablah\ub\|_{(\dB^{-\s}_{2,2})_{\beta^\perp}(\dB^{\f12}_{2,1})_\beta}
\|(\pb\uh)\om\|_{(\dB^{\s}_{2,2})_{\beta^\perp}(\dB^{-\f12}_{2,\infty})_\beta}\\
&\lesssim\f{1}{\sqrt\s}\|\ub\|_{(\dB^{1-\s}_{2,2})_{\beta^\perp}(\dB^{\f12}_{2,1})_\beta}
\|\pb\uh\|_{\dH^{1-\s,0}_\beta}\|\om\|_{\dH^{2\s,0}_\beta}\\
&\lesssim\f{1}{\sqrt\s}\|\pb\ub\|_{L^2}^{2\s}
\|\ub\|_{\dH^{\f32}}^{1-2\s}
\|\pb\uh\|_{L^2}^\s\|\nablah\pb\uh\|_{L^2}^{1-\s}
\|\om\|_{L^2}^{1-2\s}\|\nablah\om\|_{L^2}^{2\s}.
\end{align*}
Whereas by using the Helmholtz decomposition \eqref{4.2},
the $L^2$ boundness for double Riesz transform,
and the fact: $|\beta|=1$, we have
$$\|\nablah\pb\uh\|_{L^2}\lesssim \|\pb\om\|_{L^2}
+\|\pb^2\ub\|_{L^2}\leq\|\nabla\om\|_{L^2}+\|\nabla\pb\ub\|_{L^2}.$$
As a result, it comes out
\begin{equation}\begin{split}\label{5.4}
|I_3|&\leq\f{C}{\sqrt\s}\|\pb\ub\|_{L^2}^{2\s}\|\om\|_{L^2}^{1-2\s}
\bigl(\|\nabla\om\|_{L^2}+\|\nabla\pb\ub\|_{L^2}\bigr)^{1+\s}
\|\ub\|_{\dH^{\f32}}^{1-2\s}\|\pb\uh\|_{L^2}^\s\\
&\leq\f18\bigl(\|\nabla\om\|_{L^2}^2+\|\nabla\pb\ub\|_{L^2}^2\bigr)
+\f{C}{\s}\bigl(\|\om\|_{L^2}^{\f{2}{1-\s}}+\|\pb\ub\|_{L^2}^{\f{2}{1-\s}}\bigr)
\|\ub\|_{\dH^{\f32}}^{\f{2(1-2\s)}{1-\s}}\|\nabla u\|_{L^2}^{\f{2\s}{1-\s}},
\end{split}\end{equation}
where in the last step, we  used the elementary inequality that
$$\bigl(\f{1}{\sqrt\s}\bigr)^{\f{2}{1-\s}}
=\f1{\s^{\f\s{1-\s}}}\f1{\s}
\leq5^{\f14}\f{1}{\s},\quad\forall\ \s\in ]0, 1/5].$$

By substituting \eqref{5.2}-\eqref{5.4} into \eqref{5.1}, we achieve
\begin{equation}\begin{split}\label{5.5}
\f{d}{dt}\|\om\|_{L^2}^2+\f54\|\nabla\om\|_{L^2}^2
\leq&\f14\|\nabla\pb\ub\|_{L^2}^2+C\|u_0\|_{L^2}^2\bigl(|\tau'|^2+|\nu'|^2\bigr)
+C\|\om\|_{L^2}^2\|\ub\|_{\dH^{\f32}}^2\\
&+\f{C}{\s}\bigl(\|\pb\ub\|_{L^2}^{\f{2}{1-\s}}+\|\om\|_{L^2}^{\f{2}{1-\s}}\bigr)
\|\ub\|_{\dH^{\f32}}^{\f{2(1-2\s)}{1-\s}}\|\nabla u\|_{L^2}^{\f{2\s}{1-\s}}.
\end{split}\end{equation}

Similarly, by taking $L^2$ inner product of the second equation in \eqref{4.3} with $\pb\ub,$ and using the expression \eqref{4.4} for the pressure function, we obtain
\begin{equation}\begin{split}\label{5.6}
\f12\f{d}{dt}\|\pb\ub\|_{L^2}^2+&\|\nabla\pb\ub\|_{L^2}^2
=\int_{\R^3}\beta'\cdot\bigl(\pa_\beta u+\nabla\ub\bigr)\pb\ub\,dx\\
&+\int_{\R^3}\Bigl(\bigl(\pb^2\D^{-1}-1\bigr)(\pb\ub)^2
+\pb^2\D^{-1}\sum_{\ell,m\in\{\tau,\nu\}}
\pa_\ell u^m\pa_m u^\ell\Bigr)\pb\ub\,dx\\
&+\int_{\R^3}\Bigl(\bigl(2\pb^2\D^{-1}-1\bigr)
\sum_{\ell\in\{\tau,\nu\}}\pb u^\ell\pa_\ell\ub\Bigr)\pb\ub\,dx
\eqdef II_1+II_2+II_3.
\end{split}\end{equation}

Firstly, it follows from a similar derivation of  \eqref{5.2} that
\begin{equation}\begin{split}\label{5.7}
|II_1|\leq\f18\|\nabla\pb\ub\|_{L^2}^2
+C\|u_0\|_{L^2}^2|\beta'|^2.
\end{split}\end{equation}

For $II_2$, by using Sobolev embedding theorem and
the Helmholtz decomposition \eqref{4.2} together with
the $L^p~(1<p<\infty)$ boundness for double Riesz transform, we infer
\begin{equation}\begin{split}\label{5.8}
|II_2|&\leq C\bigl(\|\pb\ub\|_{L^2}\|\pb\ub\|_{L^6}
+\|\nablah\uh\|_{L^2}\|\nablah\uh\|_{L^6}\bigr)\|\pb\ub\|_{L^3}\\
&\leq C\bigl(\|\pb\ub\|_{L^2}+\|\om\|_{L^2}\bigr)
\bigl(\|\nabla\pb\ub\|_{L^2}+\|\nabla\om\|_{L^2}\bigr)
\|\nabla\ub\|_{\dH^{\f12}}\\
&\leq\f1{16}\bigl(\|\nabla\pb\ub\|_{L^2}^2+\|\nabla\om\|_{L^2}^2\bigr)
+C\bigl(\|\pb\ub\|_{L^2}^2+\|\om\|_{L^2}^2\bigr)
\|\ub\|_{\dH^{\f32}}^2.
\end{split}\end{equation}

While along the same line to the estimate of  $I_3,$
we find
\begin{align*}
|II_3|&=\Bigl|\int_{\R^3}\bigl(\pb\uh\cdot\nablah\ub\bigr)
\bigl(2\pb^2\D^{-1}-1\bigr)\pb\ub\,dx\Bigr|\\
&\leq\|\nablah\ub\|_{(\dB^{-\s}_{2,2})_{\beta^\perp}(\dB^{\f12}_{2,1})_\beta}
\bigl\|(\pb\uh)\bigl(2\pb^2\D^{-1}-1\bigr)\pb\ub\bigr\|
_{(\dB^{\s}_{2,2})_{\beta^\perp}(\dB^{-\f12}_{2,\infty})_\beta}\\
&\lesssim\f{1}{\sqrt\s}\|\ub\|_{(\dB^{1-\s}_{2,2})_{\beta^\perp}(\dB^{\f12}_{2,1})_\beta}
\|\pb\uh\|_{\dH^{1-\s,0}_\beta}\|\pb\ub\|_{\dH^{2\s,0}_\beta}.
\end{align*}
Then we get, by a similar derivation of \eqref{5.4}, that
\begin{equation}\label{5.9}
|II_3|\leq\f1{16}\bigl(\|\nabla\om\|_{L^2}^2+\|\nabla\pb\ub\|_{L^2}^2\bigr)
+\f{C}{\s}\bigl(\|\om\|_{L^2}^{\f{2}{1-\s}}+\|\pb\ub\|_{L^2}^{\f{2}{1-\s}}\bigr)
\|\ub\|_{\dH^{\f32}}^{\f{2(1-2\s)}{1-\s}}\|\nabla u\|_{L^2}^{\f{2\s}{1-\s}}.
\end{equation}

By substituting \eqref{5.7}-\eqref{5.9} into \eqref{5.6}, we conclude
\begin{align*}
\f{d}{dt}&\|\pb\ub\|_{L^2}^2+\f32\|\nabla\pb\ub\|_{L^2}^2
\leq\f14\|\nabla\om\|_{L^2}^2+C\|u_0\|_{L^2}^2|\beta'|^2\\
&+C\bigl(\|\pb\ub\|_{L^2}^2+\|\om\|_{L^2}^2\bigr)\|\ub\|_{\dH^{\f32}}^2
+\f{C}{\s}\bigl(\|\pb\ub\|_{L^2}^{\f{2}{1-\s}}+\|\om\|_{L^2}^{\f{2}{1-\s}}\bigr)
\|\ub\|_{\dH^{\f32}}^{\f{2(1-2\s)}{1-\s}}\|\nabla u\|_{L^2}^{\f{2\s}{1-\s}},
\end{align*}
from which and \eqref{5.5}, we deduce  \eqref{ineqprop1}.
This completes the proof of Proposition \ref{prop1}. \end{proof}

\section*{Acknowledgments}
Y. Liu is supported by NSF of China under grant 12101053.
Ping Zhang is supported by National Key R$\&$D Program of China under grant
2021YFA1000800 and K. C. Wong Education Foundation.
He is also partially supported by National Natural Science Foundation of China under Grants  12288201 and 12031006.


\begin{thebibliography}{50}

\bibitem{BCD} H. Bahouri, J.~Y. Chemin and R. Danchin,
{\it Fourier analysis and
nonlinear partial differential equations}, Grundlehren der
mathematischen Wissenschaften {\bf343}, Springer-Verlag Berlin
Heidelberg, 2011.

\bibitem{Bo81} J.-M. Bony, Calcul symbolique et propagation des
   singularit\'es pour les \'equations aux d\'eriv\'ees partielles non
   lin\'eaires, {\it Ann. Sci. \'Ec. Norm. Sup\'er.}, {\bf
   14} (1981), 209--246.

\bibitem{CT08} C. Cao and E.~S. Titi,
Regularity criteria for the three-dimensional Navier-Stokes equations,
{\it Indiana Univ. Math. J.}, {\bf57} (2008), 2643--2661.

\bibitem{CT11} C. Cao and E.~S. Titi,
Global regularity criterion for the 3D Navier-Stokes equations
involving one entry of the velocity gradient tensor,
{\it Arch. Ration. Mech. Anal.}, {\bf202} (2011),  919--932.

\bibitem{CW21} D. Chae and J. Wolf,
On the Serrin-type condition on one velocity component for the
Navier-Stokes equations, {\it Arch. Ration. Mech. Anal.},
{\bf 240} (2021),  1323--1347.


\bibitem{CZ07} J.-Y. Chemin and P. Zhang,
On the global wellposedness to the $3$-D incompressible
anisotropic Navier-Stokes equations,
{\it Comm. Math. Phys.}, {\bf272} (2007), 529--566.

\bibitem{CZ5}
J.-Y. Chemin and P.  Zhang, On the critical one component regularity
for 3-D Navier-Stokes system, {\it  Ann. Sci. \'Ec. Norm. Sup\'er.}
 (4), {\bf 49} (2016),   131--167.

\bibitem{CZZ}
J.-Y. Chemin, P. Zhang and Z. Zhang, On the critical one component
regularity for 3-D Navier-Stokes system: general case,
{\it Arch. Ration. Mech. Anal.},
 {\bf 224} (2017), 871--905.

\bibitem{ISS} L. Escauriaza, G. Seregin and V. \v{S}ver\'{a}k,
$L_{3,\infty}$-solutions of Navier-Stokes equations and backward
uniqueness, (Russian) {\it Uspekhi Mat. Nauk}, {\bf 58}, 2003, no.
2(350), pages 3-44; translation in {\it Russian Math. Surveys}, {\bf
58} , 2003, pages  211--250.

\bibitem{fujitakato}
H. Fujita and T. Kato, On the Navier-Stokes initial value problem I,
{\it  Arch. Ration. Mech. Anal.}, {\bf 16} (1964),
269--315.

\bibitem{LeiZhao}
B. Han, Z. Lei, D. Li and N. Zhao,
Sharp one component regularity for Navier-Stokes,
{\it Arch. Ration. Mech. Anal.}, {\bf231} (2019), 939--970.

\bibitem{KK11} C.~E. Kenig and G.~S. Koch,
An alternative approach to regularity for the Navier-Stokes equations in critical spaces,
{\it Ann. Inst. H. Poincar\'e Anal. Non Lin\'eaire}, {\bf28} (2011), 159--187.

\bibitem{Kukavica} I. Kukavica and M. Ziane,
One component regularity for the Navier-Stokes equations,
{\it Nonlinearity} {\bf19} (2006),  453--469.

\bibitem {lerayns}
J. Leray, Essai sur le mouvement d'un liquide visqueux emplissant
l'espace, {\em Acta Math.}, {\bf 63} (1933),  193--248.

\bibitem{21}
J. Neustupa and P. Penel, Regularity of a suitable weak solution to the Navier-Stokes
equations as a consequence of regularity of one velocity component, in {\it Applied
nonlinear analysis}, Kluwer/Plenum, New York, 1999, 391-402.

\bibitem{Pa02} M. Paicu, \'Equation anisotrope
de Navier-Stokes dans des espaces  critiques,
{\it  Rev. Mat. Iberoamericana,} {\bf 21} (2005), 179--235.

\bibitem{Tao} T. Tao,
 Quantitative bounds for critically bounded solutions to the Navier-Stokes equations. Nine mathematical challenges-an elucidation, 149-193, {\it Proc. Sympos. Pure Math.}, {\bf 104}, Amer. Math. Soc., Providence, RI, [2021], 2021.

\bibitem{WWZ} W. Wang, D. Wu and Z. Zhang,
Scaling invariant Serrin criterion via one velocity component
for the Navier-Stokes equations. arXiv:2005.11906 [math.AP].

\bibitem{ZP10} Y. Zhou and M. Pokorn\'{y},
On the regularity of the solutions of the Navier-Stokes equations
via one velocity component, {\it Nonlinearity}, {\bf23} (2010),  1097--1107.

\end{thebibliography}
\end{document}